\documentclass[letterpaper, 10 pt, conference]{ieeeconf}  

\IEEEoverridecommandlockouts                              
\makeatletter
\let\proof\@undefined                        
\let\endproof\@undefined                  
\makeatother 

\usepackage{cite}
\usepackage{amsmath,amssymb,amsfonts}
\usepackage{algorithmic}
\usepackage{algorithm}
\usepackage{graphicx}
\usepackage{textcomp}
\usepackage{amsmath,bm}
\usepackage{bbm}

\usepackage[font=singlespacing]{caption}
\usepackage{enumerate}

\usepackage{epsfig}
\usepackage{amsmath} 
\usepackage{amssymb}  
\usepackage{amsthm}
\usepackage{multirow}
\usepackage{bbm, dsfont}
\usepackage{bm} 
\usepackage{url}
\usepackage{balance}
\usepackage{color}
\graphicspath{{./figures/}}
\usepackage{array}
\usepackage[a]{esvect}

\newtheorem{theorem}{Theorem}
\newtheorem{definition}{Definition}

\newtheorem{lemma}{Lemma}

\newtheorem{problem}{Problem}

\usepackage{caption}
\usepackage{subcaption}

\definecolor{officegreen}{rgb}{0.0, 0.5, 0.0}

\newcommand{\hasn}[1]{{\color{black} #1}}
\newcommand{\moha}[1]{{\color{black} #1}}
\newcommand{\shen}[1]{{\color{black} #1}}
\newcommand{\yongn}[1]{{\color{black} #1}}

\title{\LARGE \bf
Incremental Affine Abstraction of Nonlinear Systems
}

\author{Syed M. Hassaan, Mohammad Khajenejad, Spencer Jensen, Qiang Shen and Sze Zheng Yong
\thanks{S.M. Hassaan, M. Khajenejad, S. Jensen and S.Z. Yong are with School for Engineering of Matter, Transport and Energy, 
Arizona State University, Tempe, AZ, USA; 
Q. Shen is with the School of Aeronautics and Astronautics, Shanghai Jiao Tong University, Shanghai, P.R. China
(email: {\tt \{shassaan,mkhajene,sjensen8,szyong\}@asu.edu, qiangshen@sjtu.edu.cn}). 
This work was supported in part by DARPA grant D18AP00073. We acknowledge Research Computing at Arizona State University (\protect\url{http://www.researchcomputing.asu.edu}) for providing High Performance Computing resources that have contributed to the results reported within this paper. }
}

\begin{document}

\maketitle

\begin{abstract}
In this paper, we propose an incremental abstraction method for dynamically over-approximating nonlinear systems in a bounded domain  by solving a sequence of linear programs, resulting in a sequence of affine upper and lower hyperplanes with expanding operating regions. \yongn{Although the affine abstraction problem can be solved offline using a single linear program, existing approaches suffer from a computation space complexity that grows exponentially with the state dimension. Hence,} the motivation for incremental abstraction is to reduce the space complexity for high-dimensional systems, but at the cost of yielding potentially worse abstractions/over-approximations.
Specifically, we start with an operating region that is a subregion of the state space and compute two affine hyperplanes that bracket the nonlinear function locally. Then, by incrementally expanding the operating region, we dynamically update the two affine hyperplanes such that we eventually yield hyperplanes that are guaranteed to over-approximate the nonlinear system over the entire domain. 
Finally, the effectiveness of the \yongn{proposed approach} 
is demonstrated using numerical examples \yongn{of high-dimensional nonlinear systems}. 

\end{abstract}


\section{Introduction}
One of the main challenges in the area of formal verification and synthesis of complex control systems is the exponential complexity of the algorithms, thus various abstraction-based methods have been proposed for complexity reduction, e.g.,  \cite{pappas2002,Tabuada2009}. The abstraction procedure computes a simpler but over-approximated system that includes all possible behaviors of the original system while preserving properties of interest. For instance, to verify that a given complex system satisfies certain properties, we can test for the desired property on the abstracted simple system, and the test result is equivalent to or sufficient for testing for the property on the original complex system. 


\emph{Literature Review.} In general, abstraction is a systematic approximation method that partitions the state space/vector field of a complex system into finite subregions, and then approximates its dynamics in each subregion by a simpler one, resulting in a hybrid system \cite{asarin2007hybridization,Asarin2003}.  \moha{Multiple abstraction approaches have been developed for several classes of systems in the literature}, including nonlinear systems \cite{Girard2012, singh2018mesh, Singh2018}, hybrid systems \cite{Alimguzhin2017}, and uncertain affine and nonlinear systems \cite{shen2019,Jin2019CDC}. \yongn{A common abstraction method uses \emph{symbolic} approaches, e.g.,  \cite{coogan2015efficient,pola2008approximately,reissig2016feedback, 
zamani2014symbolic}, based on 
discretization of the state and input spaces to obtain dynamical abstraction systems with \emph{finitely} many number of states and inputs, 
which symbolizes sets of states and inputs of the original system.  
However, the number of symbolic states and inputs 
typically grows 
exponentially with state and input dimensions.} 

\yongn{On the other hand, the work} in \cite{Girard2012} \moha{considers the over-approximation of} nonlinear \moha{vector fields with} affine systems, where the approximation error is accounted for with an additive disturbance. Further, based on a partition technique using Lebesgue integrals and 
sampling, a piecewise affine abstraction and its corresponding approximation error bounds are obtained in \cite{Azuma2010} to approximate a class of nonlinear systems with specified accuracy and relatively few 
subregions. 

In contrast to \cite{Girard2012, Azuma2010}, where a single simpler function with a bounded error term is used to abstract the original system dynamics, 
recent works \moha{in} \cite{singh2018mesh,Alimguzhin2017,shen2019,Jin2019CDC} employ 
\moha{upper and lower} affine functions to sandwich/bracket the original system dynamics, in the sense of inclusion of all possible behaviors for each subregion. In particular, \moha{the authors of} \cite{Alimguzhin2017} \moha{proposed} 
an affine abstraction approach for nonlinear Lipschitz continuous functions, resulting in two affine hyperplanes, as upper and lower bounds to bracket the original system dynamics, while 
in \cite{singh2018mesh}, two piecewise affine functions \moha{were derived} by solving a linear program for \moha{each} bounded subregion of the state space  to over-approximate 
nonlinear systems with different degrees of smoothness. 
However, \yongn{although these abstraction methods can be solved offline for each subregion using a single linear program, they have scalability issues when the original system is a high-dimensional system since the computation\moha{al} complexity grows exponentially with the state dimension}.

\emph{Contributions.} In this paper, an incremental abstraction method is proposed to dynamically over-approximate nonlinear systems \yongn{to overcome the issue of space complexity. Specifically, we propose a novel method to carry out the abstraction process sequentially, starting with a small operating region that is a subset of the entire domain and incrementally expanding to larger domains by adding new grid points, until all grid points are added. At each increment, 
a local abstraction consisting of two affine hyperplanes can be obtained by solving a linear program. 
This is in contrast to the conventional mesh-based abstraction methods, e.g., in \cite{singh2018mesh,Alimguzhin2017}, that construct abstractions statically over all grid points in the interior of the domain of interest  and have the aforementioned space complexity issues. 

Moreover, by design, our proposed incremental abstraction approach has reduced space complexity when compared to \cite{singh2018mesh,Alimguzhin2017}. The reason is that  our approach only considers the boundary points of the previous region and the newly added grid points for computing the local abstraction at each increment. More importantly, our approach provides us control over the amount of memory that is allocated to solve each linear program, and we have a rigorous proof that guarantees that the incremental abstraction is indeed an over-approximation/abstraction of the original system, which is an important feature when used for reachability analysis and robust control synthesis. 
The simulation results demonstrate that the proposed incremental approach is able to abstract high-dimensional nonlinear systems with limited space resources, but at the cost of obtaining a worse over-approximation \shen{and} a longer total computation time due to the sequence of linear programs that need to be solved. Note, however, that the time complexity is less of a concern, since the resulting linear programs are solved offline.

}

\section{Preliminaries}



For a vector $v \in \mathbb{R}^n$ and a matrix $M \in \mathbb{R}^{p \times q}$, $\|v\|_i$ and $\|M\|_i$ denote their  (induced) $i$-norm with $i=\{1,2,\infty\}$.

\subsection{Modeling Framework and Definitions} \label{sec:model}

Consider the nonlinear system:
\begin{align} \label{eq:nonl_sys}
	{x}^+ = f(x,u),
\end{align}
where $x \in \mathcal{X}=[\underline{x},\overline{x}]^n \subseteq \mathbb{R}^n$ is the system state with a bounded and closed interval domain $\mathcal{X}$, $u \in \mathcal{U}= [\underline{u},\overline{u}]^m \subseteq \mathbb{R}^m$ is the known control input with a bounded and closed interval domain $\mathcal{U}$ and vector field $f : \mathcal{X} \times \mathcal{U} \to \mathbb{R}^n$ is a continuous function. For discrete-time systems, $x^+$ denotes the state at the next time instant while for continuous-time systems, $x^+ =\dot{x}$ is the time derivative of the state. We denote $(x,u) \in \mathbb{R}^{n+m}$ a \emph{sample point} throughout the paper.

To incrementally abstract the nonlinear system \eqref{eq:nonl_sys}, we introduce the following definitions for each increment $k \in \mathbb{N}$.

\begin{definition}[Uniform Mesh and Grid Points]\label{def:grid} 
A uniform mesh of each domain $ \mathcal{X} \times \mathcal{U}$ is a collection of  $s_{\max}$ number of points, called grid points, uniformly distributed along all directions and dimensions. The set of grid points is denoted as $\mathcal{M}$ and by construction, the convex hull of $\mathcal{M}$ is the entire domain $ \mathcal{X} \times \mathcal{U}$, i.e., $\mathcal{X} \times \mathcal{U} =Conv (\mathcal{M})$. 
\end{definition}

\begin{definition}[Diameter] \label{diameter} 
The \textit{diameter} $\delta$ of each mesh element in a uniform mesh is the greatest distance between two vertices of the mesh element.
\end{definition}

\begin{definition}[Sample Set and  Operating Region] \label{def:sam_set}
At any increment $k$, a set $\mathcal{S}_k$ is called a sample set if it is a subset of all the existing grid points. Moreover, all grid points in the convex hull of the sample set is called the operating region and is denoted by $\mathcal{R}_k$, i.e., $\mathcal{R}_k \triangleq Conv(\mathcal{S}_k) \cap \mathcal{M}$. 
\end{definition}
\begin{definition}[Expanding Operation Region] \label{def:exp_dom}
	At each increment $k$, the operation region $\mathcal{R}_k$ is expanding if $\mathcal{R}_{k-1} \subset \mathcal{R}_{k}$, i.e., the new operating region at the current increment is a strict superset of the previous operating region. 
\end{definition}
\begin{definition}[Vertex Set] \label{def:bndry_pts}
	Given an operating region $\mathcal{R}_k$ at increment $k$, the set of all vertices of the convex hull of $\mathcal{R}_k$ is called the vertex set, 
	and denoted as $\mathcal{V}_k \triangleq Ver(Conv(\mathcal{R}_k))$. Note that the convex hull of the operating region is a polytope and has a well-defined vertex set. 
\end{definition}\vspace{-0.05cm}

The process of over-approximating a nonlinear function as given in (\ref{eq:nonl_sys}) can be defined as follows, \yongn{similar} to \cite{singh2018mesh}: 
\begin{definition}[Affine Abstraction Model] 
\label{def:aff_abstr}
	Given a bounded-domain function $f(x,u)$, the affine functions $\overline{f}(x,u) = \overline{A} x + \overline{B} u + \overline{h}$ and $\underline{f}(x,u) = \underline{A} x + \underline{B} u + \underline{h}$, are called upper and lower \yongn{affine functions of} $f(x,u)$, respectively, if $\forall (x,u) \in \mathcal{X} \times \mathcal{U}$, $\underline{f}(x,u) \leq f(x,u) \leq \overline{f}(x,u)$. 
The pair of functions $\mathcal{F} \triangleq \{\overline{f}(x,u), \underline{f}(x,u)\}$ forms an affine abstraction model that over-approximates the given function $f(x,u)$.
\end{definition}\vspace{-0.05cm}

 One major goal when finding affine abstractions is to get them as tight as possible with \yongn{a low abstraction error, i.e., with a small} distance between the affine hyperplanes: 
\begin{definition}[Abstraction Error \cite{singh2018mesh}]\label{def:abstr_err}
	The abstraction error of an affine abstraction model $\mathcal{F}$ of a nonlinear function $f(x,u)$ over its domain $\mathcal{X} \times \mathcal{U}$, at increment $k$, is defined as $\theta = \max_{(x,u) \in \mathcal{X} \times \mathcal{U}} ||\overline{f}(x,u) - \underline{f}(x,u)||_1$. 
\end{definition}

\shen{
\yongn{Next, we reproduce a lemma from \cite{Stampfle2000} that we will rely on to find linear interpolation error bounds over mesh elements:}
\begin{lemma}[\hspace{-0.02cm}{\cite[Theorem 4.1  \& Lemma 4.3]{Stampfle2000}}] \label{lemma_accuracy}
Let $S$ be an $(n+m)$-dimensional mesh element such that $S \subseteq \mathcal{M} \subseteq \mathbb{R}^{n+m}$ 
with  diameter $\delta$ {(see Definition \ref{diameter})}. Let $f:S \rightarrow \mathbb{R}$ be a nonlinear function and let $f_{l}$ be the linear interpolation of 
$f(.)$ \yongn{evaluated} at the vertices of the mesh element $S$. Then, the approximation error bound $\sigma$ defined as the maximum error between $f$ and $f_{l}$ on $S$, i.e., 
$\sigma = \max_{s \, \in S} (|f(s) - f_{l}(s)|)$, 
is upper-bounded by 
\begin{enumerate}[(i)]
\item $\sigma \leq 2\lambda \delta_{s}$, if $f \in C^{0}$ on $S$,
\item $\sigma \leq \lambda \delta_{s}$, if $f$ is Lipschitz continuous on $S$,
\item $\sigma \leq \delta_{s} \max_{s\in S}\|f'(s)\|_{2}$, if $f \in C^{1}$ on $S$,
\item $\sigma \leq \frac{1}{2} \delta_{s}^{2} \max_{s\in S} \|f''(s)\|_{2}$, if $f \in C^{2}$ on $S$,
\end{enumerate}
where $\lambda$ is the Lipschitz constant, $f'(s)$ is the Jacobian of $f(s)$, $f''(s)$ is the Hessian of $f(s)$ and $\delta_{s}$ satisfies
\begin{align*}
\delta_{s} \leq \sqrt{\frac{n+m}{2(n+m+1)}} \delta.
\end{align*}
\end{lemma}\vspace{-0.1cm}
}

\section{Problem Formulation}

For the nonlinear function defined in \eqref{eq:nonl_sys}, previous works in  \cite{Alimguzhin2017, singh2018mesh} have proposed several different methods to find its affine abstraction. One major problem with these approaches is that they do not scale well with the number of grids. For systems where there are a very large number of grid points, which is usually the case with higher dimensional systems, the amount of memory required to store and process these points increases exponentially. Although reducing the number of grid points could solve the problem of memory consumption, it also results in poor/conservative abstractions or over-approximations. The following formalizes our notion of limited memory resources in this case. 
\begin{definition}[Maximum Number Of Points]
Limited memory resources can be expressed in terms of the limit on maximum number of points, denoted as $\overline{s}$, that can be processed at any time. Thus, for a user-specified $\overline{s}$, the total number of increments, denoted as $\kappa$, required to process all the grid points $s_{\max}$, can then be computed as:
		\begin{align} \label{eq:kappa_gen}
		\kappa = \frac{s_{max}-\overline{s}}{\overline{s}-\delta} + 1,
		\end{align}
		where $\delta$ is the number of points carried over to $\mathcal{R}_k$ from $\mathcal{R}_{k-1}$. In Section \ref{sec:result}, we will remark on the choice of $\delta$. 
\end{definition}

\vspace{-0.05cm}
Given a user specified $\overline{s}$ (i.e., when memory resources are scarce), 
one 
way to obtain a sufficiently tight affine abstraction 
is by incrementally obtaining over-approximations over smaller subregions of the domain $\mathcal{X} \times \mathcal{U}$ of $f(x,u)$ over $\kappa$ total increments. The final abstraction can then be obtained combining the incremental results to get the abstraction over the entire domain of $f(x,u)$. With this in mind, we now define the notion of incremental abstraction at increment $k$: 

\vspace{-0.1cm}
\begin{definition}[Incremental Abstraction] \label{def:incr_abstr}
	At each increment $k$, for a function $f(x,u)$ as defined in \eqref{eq:nonl_sys} with an operating region $\mathcal{R}_k$, the incremental abstraction is the affine abstraction of $f(x,u)$ over the operating region $\mathcal{R}_k$. The resulting affine hyperplanes that over-approximate $f(x,u) \, \forall (x,u) \in \mathcal{S}_k$ are denoted as $\mathcal{F}_k = \{\overline{f}_k(x,u), \underline{f}_k(x,u)\}$.
\end{definition}\vspace{-0.05cm}

Moreover, the abstraction error at each increment as well as the overall abstraction error is defined as follows: 

\vspace{-0.1cm}
\begin{definition}[Incremental Abstraction Error]
	At each increment $k$, the abstraction error of $\mathcal{F}_k$ is $\theta_k = \max_{(x,u) \in \mathcal{V}_k} ||\overline{f}_k(x,u) - \underline{f}_k(x,u)||_1$. The overall abstraction error after all $\kappa$ increments is then 
	$\theta = \max(\{\theta_i\}_{i=1}^{\kappa})$.
\end{definition}\vspace{-0.05cm}

Using the concept of incremental abstraction, the problem of affine abstraction of the system in 
\eqref{eq:nonl_sys} can be recast as: 
 
 \vspace{-0.1cm}
\begin{problem}[Affine Abstraction of a \yongn{High-Dimensional} System] \label{problem2}
	Given a 
	\yongn{high-dimensional} nonlinear function in \eqref{eq:nonl_sys}, along with the requirement that at most $\overline{s}$ samples can be taken into consideration at each increment, find the affine abstraction $\mathcal{F}$ of $f$ over $\mathcal{X} \times \mathcal{U}$ using $\{\mathcal{F}_k\}, \, \forall k \in \{1, \ldots, \kappa\}$ obtained from incremental abstractions over $\kappa$ increments, each with at most $\overline{s}$ samples, such that:
	\begin{equation} \label{eq:problem2}
	\begin{aligned}
		& \text{minimize: } \theta_k \\
		& \text{s.t.: } \overline{f}_k(x,u) \geq f(x,u) \geq \underline{f}_k(x,u), \forall (x,u) \in \mathcal{R}_k ,
			\end{aligned}
	\end{equation}
 $\forall k \in \{1, \ldots, \kappa\}$,  and $\mathcal{R}_k$ is expanding from $\mathcal{R}_{0} = \emptyset$ to $\mathcal{R}_{\kappa} = \mathcal{M}$, 
 i.e., $\emptyset =\mathcal{R}_0 \subset \mathcal{R}_1 \subset \hdots \subset \mathcal{R}_\kappa$. Then, using these incremental abstractions, find an affine abstraction over the entire domain $\mathcal{X} \times \mathcal{U} = Conv(\mathcal{R}_\kappa) =Conv(\mathcal{M})$.

\end{problem}\vspace{-0.05cm}

\yongn{Note that throughout this paper, we consider affine abstraction models with only a single region. The results in this paper also applies in a straightforward manner when the total domain $\mathcal{X} \times \mathcal{U}$ is partitioned into $p$ subdomains, as was done in the literature, e.g., \cite{Alimguzhin2017,Singh2018,singh2018mesh}, to further decrease abstraction errors, resulting in  piecewise affine abstractions.}

\vspace{-0.1cm}
\section{Main Results} \label{sec:result}\vspace{-0.05cm}
{To overcome the limitations on space complexity, we propose an incremental abstraction approach, in which at each increment, at most $\overline{s}$ number of sample points are processed to obtain an affine abstraction.}

%
%
%
\begin{lemma} \label{lem:optm}
Given \moha{the affine abstraction model} 
$\mathcal{F}_{k-1}=\{ \underline{f}_{k-1}(x,u), \overline{f}_{k+1} (x,u) \}$ \moha{for the nonlinear function $f(x,u)$} 
over an operating region $\mathcal{R}_{k-1}$, 
at increment $k$, solving the following minimization problem over the sample set $\mathcal{S}_k = (\mathcal{R}_k \setminus \mathcal{R}_{k-1}) \cup \mathcal{V}_{k-1}$, where $\mathcal{V}_{k-1} \moha{\triangleq} Ver(Conv(\mathcal{R}_{k-1}))$, obtains a functional over-approximation of $f(x,u)$ over $\mathcal{R}_k$:
\begin{subequations}
	\begin{align*} \label{eq:opt_prob}
		\underset{\theta_k, \overline{A}_k, \underline{A}_k, \overline{B}_k, \underline{B}_k, \overline{h}_{k},\underline{h}_{k}}{\text{min}} \theta_k \tag{4}
	\end{align*}
	subject to: \\
	$\forall (x,u) \in \mathcal{R}_k \setminus \mathcal{R}_{k-1}:$
	\begin{align} \label{eq:samp_constr}
	\begin{aligned}
		&\overline{A}_k\, {x} + \overline{B}_k \, {u} + \overline{h}_{k} \geq f({x},{u}), \\
		&\underline{A}_k\, {x} + \underline{B}_k \, {u} + \underline{h}_{k} \leq f({x},{u}),
	\end{aligned}
	\end{align}
	$\forall (x,u) \in \mathcal{V}_{k-1}: $
	\begin{align} \label{eq:extension_constr}
	&\begin{aligned}
		& \overline{A}_k\, {x} + \overline{B}_k \, {u}+ \overline{h}_{k} \geq \overline{A}_{k-1}\, {x} + \overline{B}_{k-1} \, {u} + \overline{h}_{k-1},  \\
		& \underline{A}_k\, {x} + \underline{B}_k \, {u} + \underline{h}_{k} \leq \underline{A}_{k-1}\, {x} + \underline{B}_{k-1} \, {u} + \underline{h}_{k-1},
	\end{aligned}
	\end{align} 
	$\forall (x,u) \in \mathcal{V}_k =Ver(Conv(\mathcal{S}_{k})):$
	\begin{align} \label{eq:error_constr}
	\begin{aligned}
		& (\overline{A}_k - \underline{A}_k) \, x + (\overline{B}_k - \underline{B}_k) \, u + \overline{h}_{k} - \underline{h}_{k} \leq \theta_k \mathds{1}_n.
	\end{aligned}
	\end{align}
\end{subequations}
\end{lemma}

\begin{proof}
In the optimization problem given in \eqref{eq:opt_prob}, the constraints \eqref{eq:samp_constr} and \eqref{eq:extension_constr} make sure that the two hyperplanes at increment $k$ bracket the nonlinear function for all newly added grid points and the vertices of operating region $\mathcal{R}_{k-1}$, respectively. 
Moreover, in light of \cite[Lemma 1]{singh2018mesh}, it is obtained from \eqref{eq:extension_constr} that $\forall (x,u) \in \mathcal{R}_{k-1}$,
	\begin{align} 
	\begin{aligned} \label{eq:abstraction_bb} 
		 \overline{f}_{k}(x,u)  \geq \overline{f}_{k-1}(x,u),  \; \underline{f}_{k}(x,u)  \leq \underline{f}_{k-1}(x,u).
	\end{aligned}
	\end{align}
	Since the given two affine hyperplanes $\mathcal{F}_{k-1}=\{ \underline{f}_{k-1}(x,u), \overline{f}_{k-1} (x,u) \}$ over-approximate the nonlinear function over operating region $\mathcal{R}_{k-1}$, i.e, $\underline{f}_{k-1}(x,u) \le f(x,u) \le \overline{f}_{k-1}(x,u)$, $\forall (x,u) \in \mathcal{R}_{k-1}$, we further have
	\begin{align} \label{eq:abstraction_bbb}
	\underline{f}_{k}(x,u) \le f(x,u) \le \overline{f}_{k}(x,u), \quad \forall (x,u) \in \mathcal{R}_{k-1}.	
	\end{align}
As a result, it follows from \eqref{eq:samp_constr}  and \eqref{eq:abstraction_bbb} that 
	\begin{align} 
	\underline{f}_{k}(x,u) \leq f(x,u) \leq \overline{f}_{k}(x,u),  \quad \forall (x,u) \in \mathcal{R}_{k},\label{eq:k+1_final}  
	\end{align}
which implies that the affine hyperplanes obtained at increment $k$ over-approximate the nonlinear function $f(x,u)$ overall the current operating region $\mathcal{R}_k$.

Finally, the constraint in \eqref{eq:error_constr} ensures that the two affine hyperplanes obtained at the increment $k$ are as close to each other as possible, \yongn{i.e., the abstraction error is minimized}.
\end{proof}\vspace{-0.15cm}


Using the above lemma, \yongn{we prove in the following theorem that incremental affine abstraction also yields an affine abstraction model of the system in \eqref{eq:nonl_sys}, solving Problem \ref{problem2}.}
\begin{theorem} \moha{\label{thm:incabs}}
	Consider the nonlinear system \eqref{eq:nonl_sys} with $(x,u) \in \mathcal{X} \times \mathcal{U}$. 
	Let $\overline{s}$ indicate the maximum number of sample points allowed to be taken at each iteration $k$. Algorithm \ref{algorithm:incabs} incrementally solves the abstraction problem formulated in Problem \ref{problem2}, i.e., $\forall (x,u) \in \mathcal{X}\times\mathcal{U}$, it returns upper and lower affine functions $\overline{f}(x,u)\yongn{=\overline{f}_k(x,u)+\sigma \mathds{1}}$ and $\underline{f}(x,u)\yongn{=\underline{f}_k(x,u)-\sigma \mathds{1}}$ that over-approximate the nonlinear system \eqref{eq:nonl_sys}, \yongn{with the corresponding interpolation error $\sigma$ in Lemma \ref{lemma_accuracy} and $\mathds{1}$ is a vector of ones}. 
\end{theorem}\vspace{-0.375cm}


\begin{proof}
	Using mathematical induction, we will prove that Theorem \ref{thm:incabs} solves the Problem \ref{problem2} incrementally. 
	
	In the first increment $k=1$, we have the operating region $\mathcal{R}_1$. Since $\mathcal{R}_0  = \emptyset$, we have $\mathcal{V}_0 = \emptyset$. Therefore, we further have $\mathcal{S}_1 = (\mathcal{R}_1 \setminus \mathcal{R}_{0}) \cup \mathcal{V}_{0} = \mathcal{R}_1$.
	Based on Algorithm \ref{algorithm:incabs}, solving the optimization problem defined in Lemma \ref{lem:optm} over $\mathcal{S}_1$ will yield the affine hyperplanes $\mathcal{F}_1 = \{\underline{f}_1(x,u),\overline{f}_1(x,u)\}$ with:
	\begin{align}\vspace{-0.05cm}
		\nonumber \underline{f}_1(x,u) = \underline{A}_1 x + \underline{B}_1 u + \underline{h}_1, \;
		\overline{f}_1(x,u) = \overline{A}_1 x+ \overline{B}_1 u + \overline{h}_1.
	\end{align}
	Since $\mathcal{S}_1 =\mathcal{R}_1$, these two hyperplanes also bracket the function $f(x)$ at all sample points in $\mathcal{R}_1$, i.e.
	\begin{align*}\vspace{-0.05cm}
		\underline{f}_1 (x,u) \le f(x,u) \le \overline{f}_1(x,u), \quad \forall (x,u) \in \mathcal{R}_1.
	\end{align*}
	
	
	At increment $k>1$, suppose that the obtained affine hyperplanes $\mathcal{F}_k = \{ \underline{f}_k(x,u), \overline{f}_k(x,u) \}$ over $(x,u) \in \mathcal{S}_k = (\mathcal{R}_k \setminus \mathcal{R}_{k-1}) \cup \mathcal{V}_{k-1}$ satisfy:
	\begin{align*} 
	\underline{f}_k(x,u) \le f(x,u) \le \overline{f}_k(x,u), \quad \forall (x,u) \in \mathcal{R}_k.
	\end{align*}

Then, follow the same lines in the proof of Lemma \ref{lem:optm} for increment $k+1$, we have
	\begin{align*} 
	\underline{f}_{k+1}(x,u) \leq f(x,u) \leq \overline{f}_{k+1}(x,u),  \ \forall (x,u) \in \mathcal{R}_{k+1}.
	\end{align*}
	
%
	
 	Therefore, the affine hyperplane obtained at any future increment will also over-approximate the nonlinear function over all the past operating regions, hence at the last increment $k=\kappa$, the final two affine hyperplanes $\mathcal{F}_{\kappa} = \{\underline{f}_{\kappa}(x,u), \overline{f}_{\kappa}(x,u) \}$ will over-approximate the nonlinear function over the entire mesh since the operating region $\mathcal{R}_{\kappa} =Conv(\mathcal{S}_{\kappa}) \cap \mathcal{M} = \mathcal{M}$ contains all $s_{max}$ samples. Finally, using \yongn{a combination of} the result in \cite[Lemma 2]{singh2018mesh} \yongn{and Lemma \ref{lemma_accuracy}}, the \yongn{desired} affine abstraction \yongn{can be obtained by accounting for the interpolation errors when extending from grid points of the mesh to the entire continuous domain (cf. step 5 of Algorithm \ref{algorithm:incabs})}.
	This completes the proof. 
%
\end{proof}\vspace{-0.2cm}


\begin{algorithm}[t]
\caption{Procedures of Incremental Abstraction}  \label{algorithm:incabs}
 \begin{enumerate}
	\item Initialize $k=1$, $\mathcal{R}_0 = \emptyset \implies \mathcal{V}_0 = \emptyset$.
	\item At increment $k$, consider a new sample set $\mathcal{S}_k = (\mathcal{R}_k \setminus \mathcal{R}_{k-1}) \cup \mathcal{V}_{k-1}$ of size $ \overline{s} $, where the set $(\mathcal{R}_k \setminus \mathcal{R}_{k-1}) \neq \emptyset$ denotes the newly added grid points such that $\mathcal{R}_k$ is expanding with $k$.
	\item For the sample set $\mathcal{S}_k$, use Lemma \ref{lem:optm} to obtain hyperplanes $\mathcal{F}_k = \{\overline{f}_k,\underline{f}_k\}$ that over-approximate the nonlinear function \eqref{eq:nonl_sys} over $\mathcal{S}_k$.
	\item Go to step 2 with $k=k+1$ if $k < \kappa$.
	\item After obtaining the final hyperplanes $\mathcal{F}_{\kappa} = \{\underline{f}_{\kappa}(x,u), \overline{f}_{\kappa}(x,u) \}$, the affine abstraction over the domain $\mathcal{X}\times \mathcal{U}$ for the system \eqref{eq:nonl_sys} is: 
	\begin{align*} 
	\begin{aligned}
	\overline{f}(x,u) = \overline{A}_{\kappa} x +  \overline{B}_{\kappa} u + \overline{h}_{\kappa} + \sigma, \\
	\underline{f}(x,u) = \underline{A}_{\kappa} x +  \underline{B}_{\kappa} u + \underline{h}_{\kappa} -\sigma,
	\end{aligned}
	\end{align*}
	where $\sigma$ is the approximation error 
	in \shen{Lemma \ref{lemma_accuracy}}.
	\end{enumerate}
\end{algorithm}

To reduce space complexity, the proposed incremental abstraction \yongn{algorithm} only computes affine hyperplanes for $\overline{s}$ sample points at each increment $k$. As shown in step 2 of the Algorithm \ref{algorithm:incabs}, at each increment $k$, we consider a new sample set $\mathcal{S}_k = (\mathcal{R}_k \setminus \mathcal{R}_{k-1}) \cup \mathcal{V}_{k-1}$ of size $ \overline{s} $ and discard the previously processed points from the set $\mathcal{R}_{k-1} \setminus \mathcal{V}_{k-1}$ to accommodate new points. Then, in Lemma \ref{lem:optm}, we show that retaining these $\overline{s}$ grid points at each increment $k$ is enough to provide conservative over-approximation over all other discarded points at $k-1$.

	Bounds on the total number of increments $\kappa$ of the incremental abstraction can be calculated if $\overline{s}$ is given. For a  state-input domain $\mathcal{X} \times \mathcal{U} \subset \mathbb{R}^{n+m}$, in general at least $n+m+1$ grid points are required to define a hyperplane. Moreover, since we require the operating region to expand with each increment, so $\delta$, the maximum number of points that can be carried over future increments cannot exceed $\overline{s}-1$. Therefore, $\delta$ is bounded by $\delta \in \left[n+m+1,\overline{s}-1\right]$. Hence, using \eqref{eq:kappa_gen}, the following bounds on $\kappa$ apply:
	\begin{align*}
		\kappa \in \left[\frac{s_{max}-\overline{s}}{\overline{s}-(n+m+1)}+1,s_{max}-\overline{s}+1\right].
	\end{align*}
\vspace{-0.1cm}
\section{Examples and Discussion} \label{sec:Examples}\vspace{-0.05cm}
In this section, \yongn{we demonstrate} the capability of the proposed incremental abstraction approach in the limited resource setting \yongn{using two high-dimensional nonlinear systems.}

\yongn{\vspace{-0.1cm}
\subsection{Nonlinear Rastrigin's function \cite{pohlheim2005geatbx}}\vspace{-0.05cm}

\begin{figure*}[!t]
	\centering
	\begin{subfigure}[t]{0.31\textwidth}
		\centering
		\includegraphics[scale=0.12,trim=18mm 3mm 0mm 0mm]{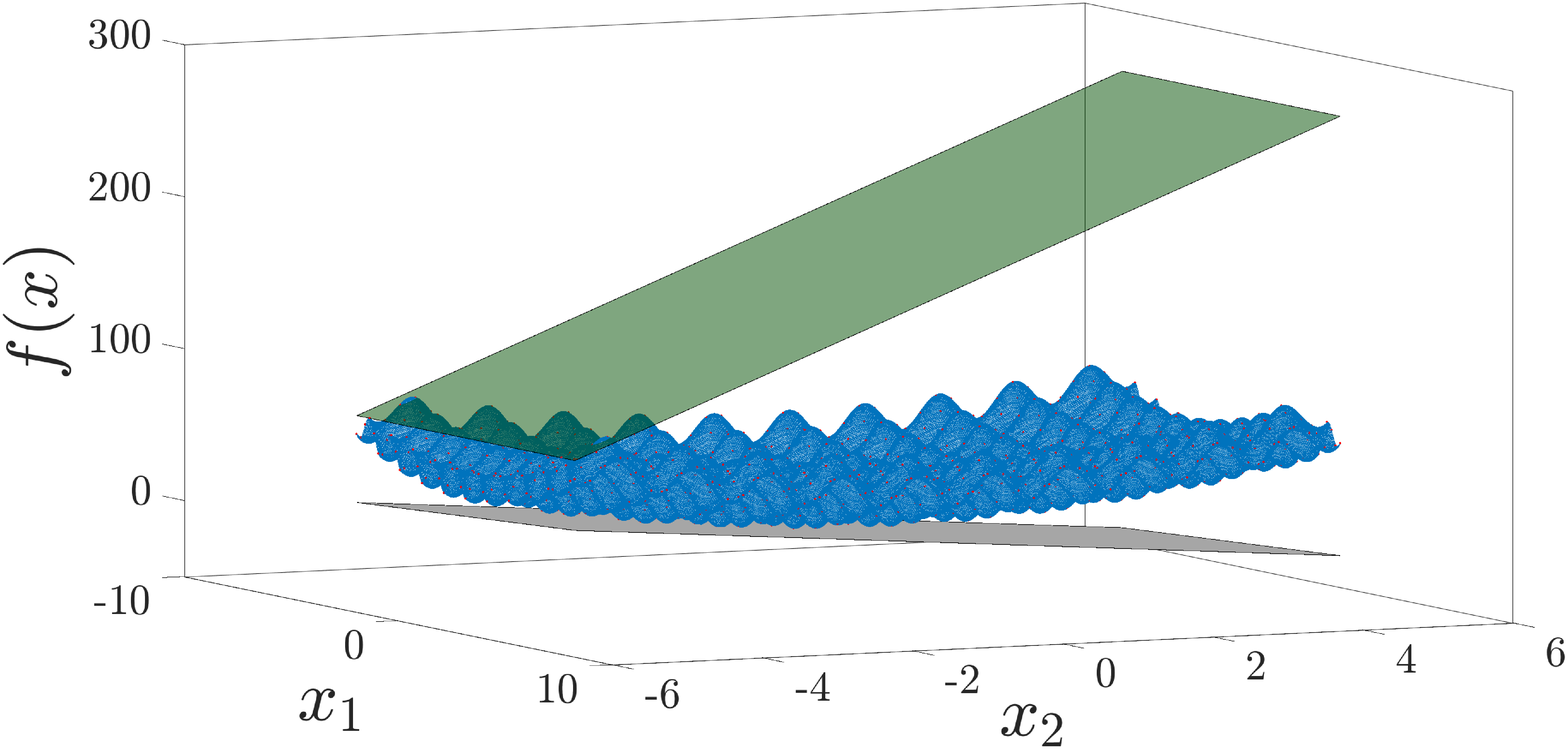}
		\caption{Abstraction with $\overline{s}=50$}
	\end{subfigure}~~
	\begin{subfigure}[t]{0.31\textwidth}
		\centering
		\includegraphics[scale=0.12,trim=12mm 3mm 0mm 0mm,clip]{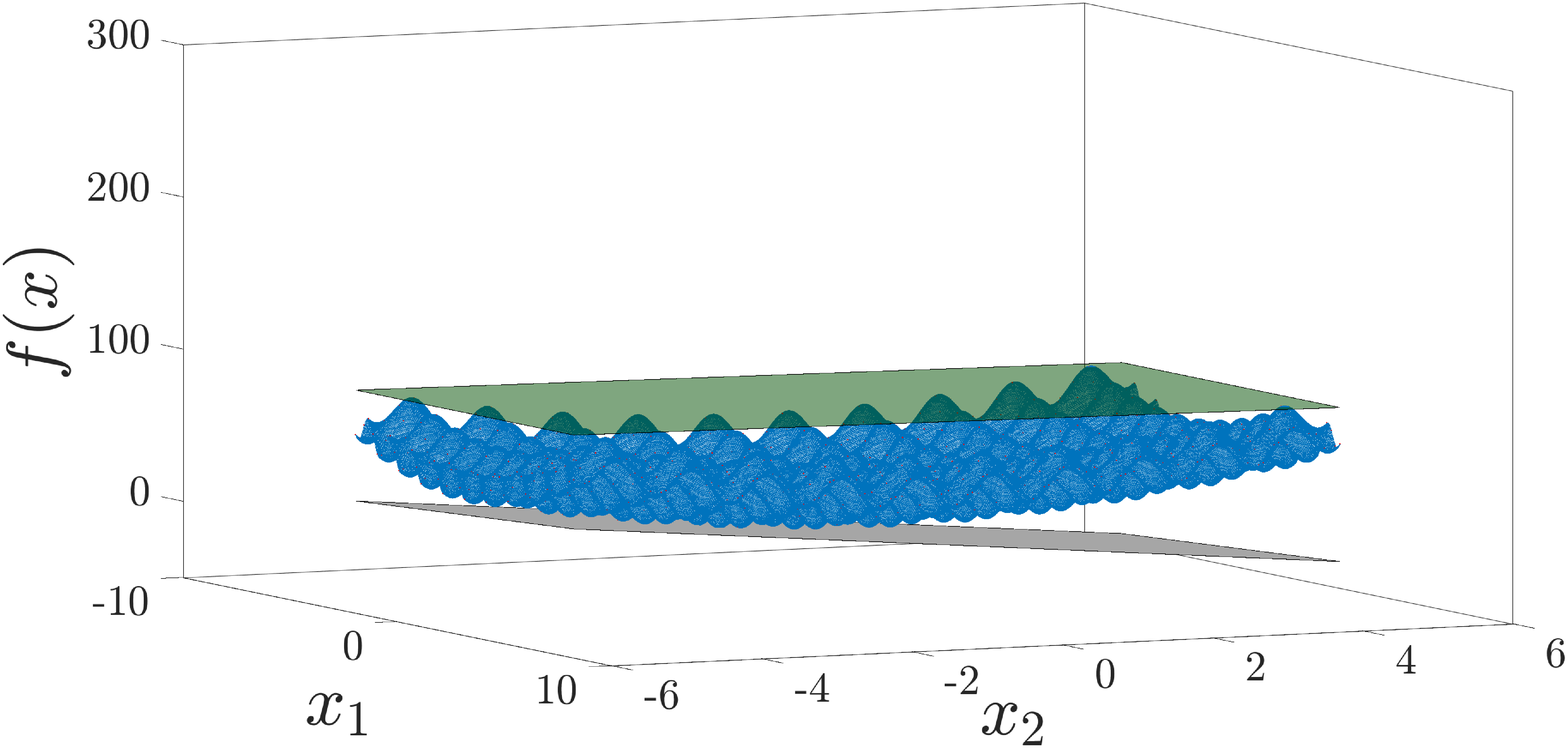}
		\caption{Abstraction with $\overline{s}=500$}
	\end{subfigure}~~
	\begin{subfigure}[t]{0.31\textwidth}
		\centering
		\includegraphics[scale=0.12,trim=15mm 3mm 10mm 0mm,clip]{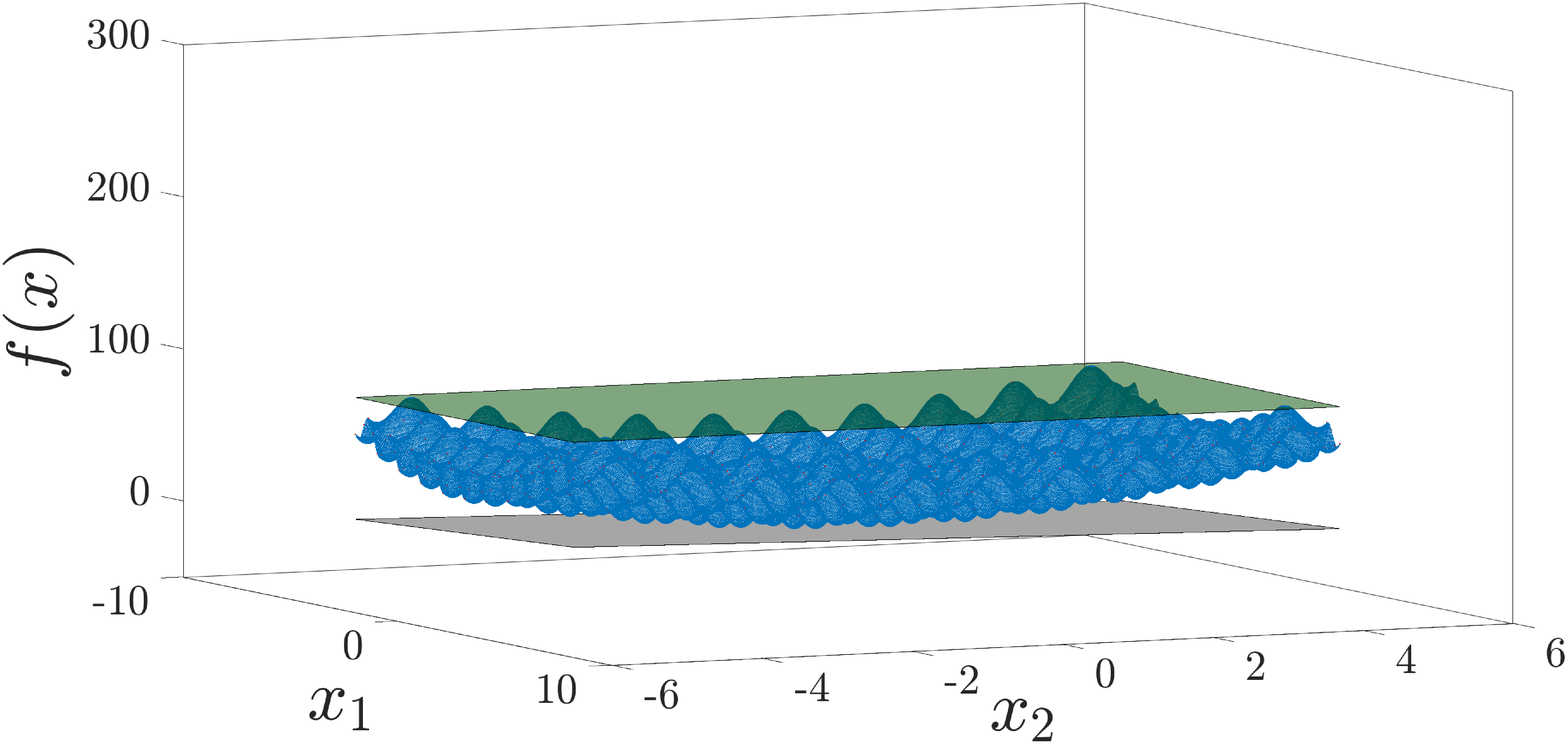}
		\caption{Abstraction with all points (same as \cite{singh2018mesh})}
	\end{subfigure}\vspace{-0.1cm}
	\caption{\yongn{Comparison of abstractions for varying maximum numbers of grid points $\bar{s}$ 
	(memory allocation)} of $(\ref{eq:rastrigin})$ with $d=2$.\vspace*{-0.5cm}\label{fig:Comp1}}
\end{figure*}
\begin{figure*}
	\centering
	\begin{subfigure}[t]{0.32\textwidth}
		\centering
		\includegraphics[scale=0.27,trim=10mm 11mm 0mm 0mm]{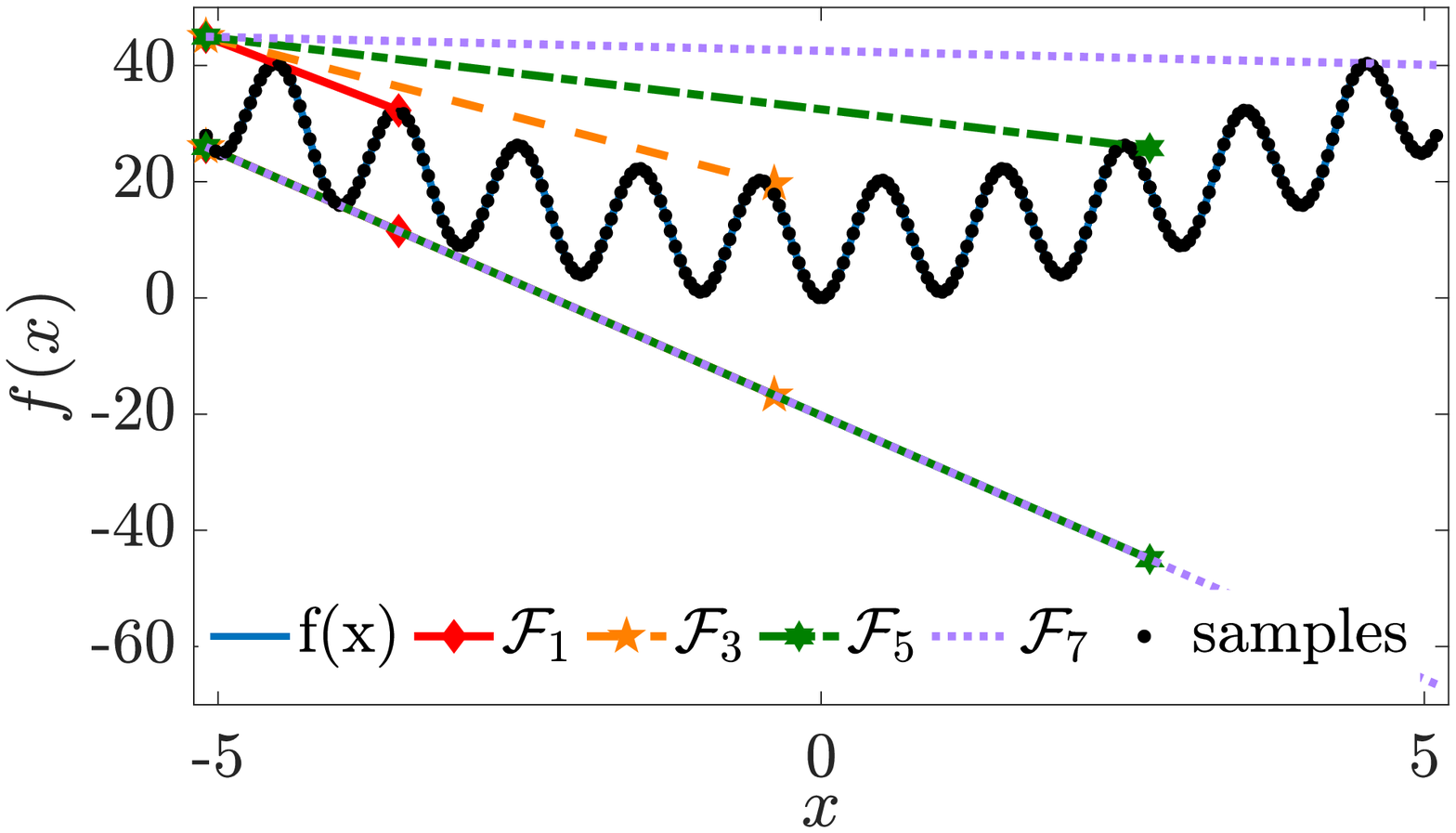}
		\caption{
		With $\mathcal{R}_1$ to the left}\label{fig:Base_IA}
	\end{subfigure}~
	\begin{subfigure}[t]{0.32\textwidth}
		\centering
		\includegraphics[scale=0.27,trim=10mm 11mm 0mm 0mm]{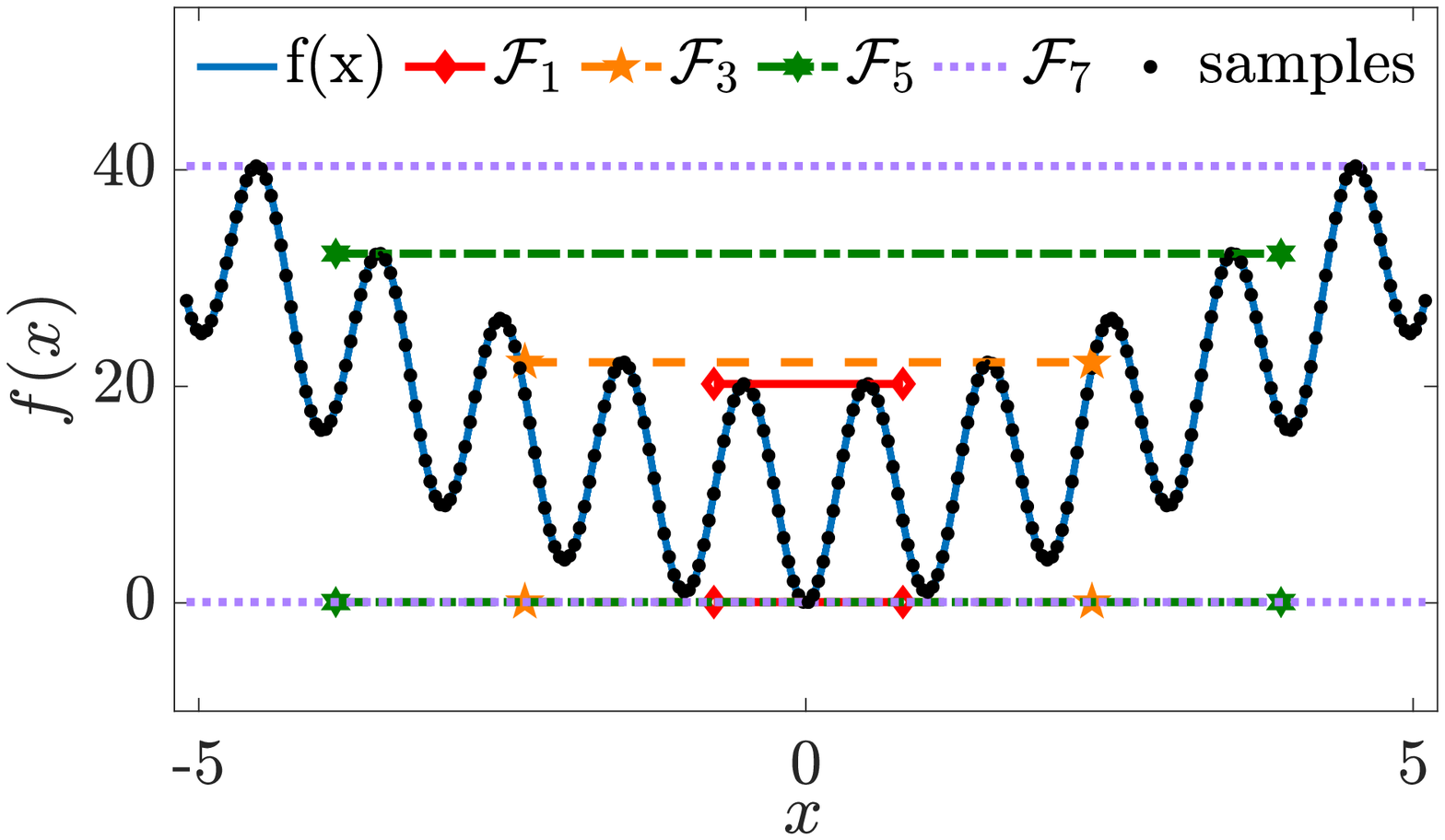}
		\caption{
		With $\mathcal{R}_1$ at the center}\label{fig:Centered_IA}
	\end{subfigure}~
	\begin{subfigure}[t]{0.32\textwidth}
		\centering
		\includegraphics[scale=0.27,trim=10mm 11mm 0mm 0mm]{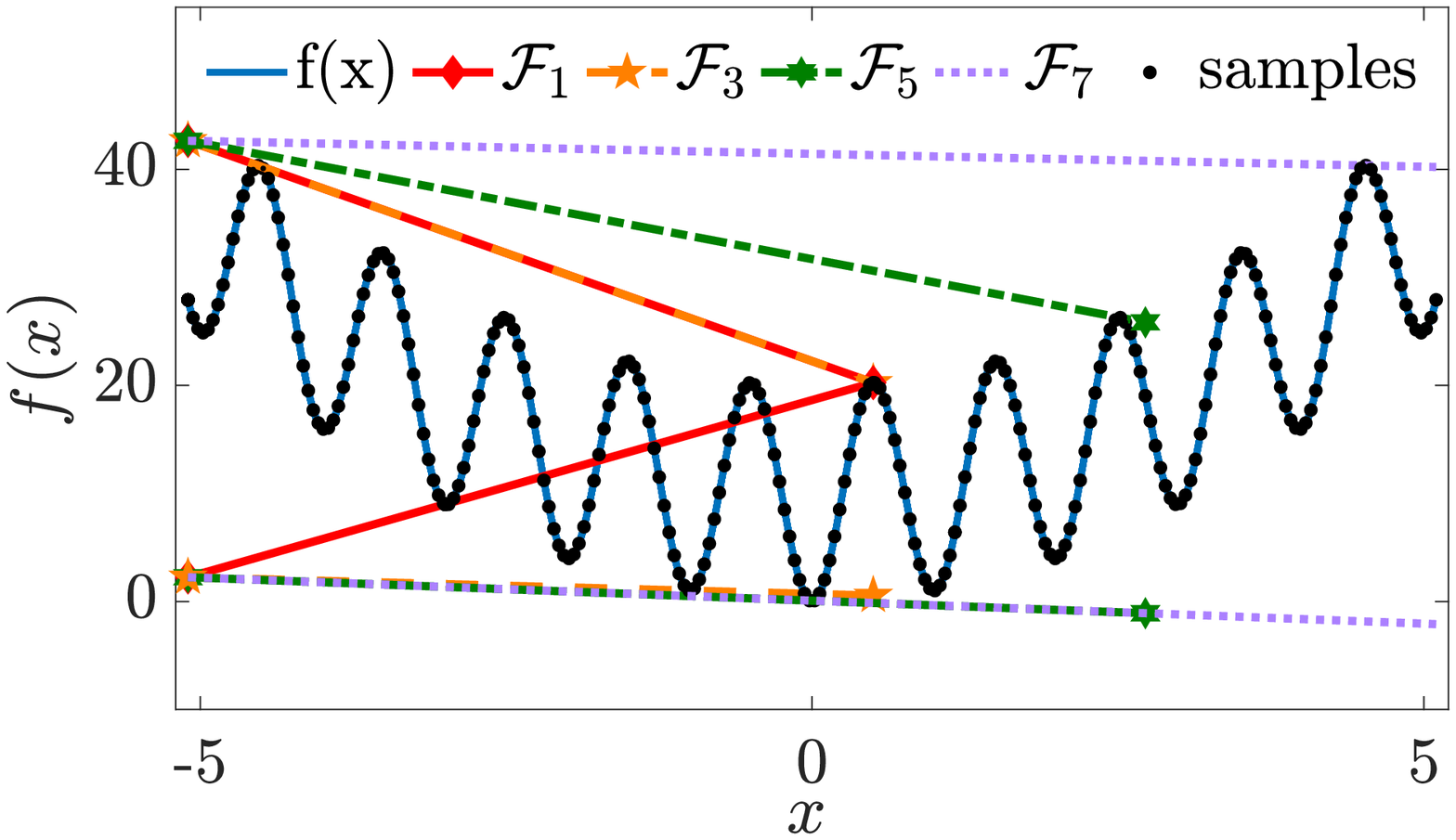}
		\caption{
		With grid point $x = 0.5$ as warm-start}\label{fig:WS_IA}
	\end{subfigure}\vspace{-0.1cm}
	\caption{\yongn{Comparison of affine abstractions of $(\ref{eq:rastrigin})$ with $d=1$ for different heuristics}. The hyperplanes in $\mathcal{F}_k$ for $k=1,3,5,7$ show the evolution of the abstraction after respective increments. The lengths of each $\mathcal{F}_k$ vary as the domain varies.\vspace*{-0.75cm}\label{fig:Comp2}}
\end{figure*}

First, we consider a nonlinear system with dynamics described by Rastrigin's function \cite{pohlheim2005geatbx}:} 
\begin{align} \label{eq:rastrigin}
\dot{x}_i= f(x) = 10d + \textstyle\sum_{j=1}^{d}[x_j^2 - 10\cos(2\pi x_j)]
\end{align}
where $x= [x_1, \ldots, x_d]^T \in \mathbb{R}^{d}$ with $d$ being the dimension of state $x$. In addition, we also assume that $x_i \in [-5.1, 5.1]$ for all $i\in\{1, \ldots, d\}$. 
All simulations are performed on Arizona State University's Agave Cluster on a single thread of one of the cores of Intel Xeon E5-2680 v4 CPU processor running at 2.40GHz. The script is written and run on MATLAB$\textsuperscript{\textregistered}$ version 2017a, \yongn{and uses Gurobi \cite{gurobi} as the linear program solver.}  
The amount of RAM available for the simulations 
is also adjusted to cater to the required environment for the 
sake of a fair comparison. Moreover, for incremental abstraction, the maximum number of grid points $\overline{s}$ that are considered in each linear program is a controllable parameter, which  we 
also vary 
for comparison. 

\paragraph{\yongn{Effects of sample size on abstraction error performance with unlimited memory}} \yongn{For our first study, we emulate} a virtually unlimited resource environment 
by setting the maximum available system RAM to 64GB, and use 
the function \eqref{eq:rastrigin} 
with a 2-dimensional domain. 
In each dimension, we consider 51 points, resulting in a total of $51^2=2601$ grid points. The computational times are compared for different cases of maximum number of grid points that can be \yongn{considered for each linear program}. In the first case, $\overline{s} = 50$ is chosen, which takes 57 increments to find the over-approximation of the 2-dimensional nonlinear system. For the second case, $\overline{s} = 500$ solves the problem in 6 increments. Finally, the last case considers all the points at once, as in \cite{singh2018mesh}, to solve the problem. Figure \ref{fig:Comp1} depicts the 
resulting lower and upper 
affine hyperplanes as well as the original nonlinear function under these three cases. \yongn{In all cases, the nonlinear system is over-approximated by the affine hyperplanes obtained from the proposed abstraction method.} Table \ref{tab:Comp1} shows the computational times for each case and the corresponding maximum distances between the 
hyperplanes, \yongn{which demonstrates that the proposed incremental abstraction is suboptimal when compared to 1-step abstraction approaches in \cite{singh2018mesh,Alimguzhin2017} and its performance in terms of abstraction error and total time is dependent on the amount of allocated memory in terms of $\bar{s}$. Therefore, taking $\overline{s}$ as a controllable parameter, the proposed abstraction method allows the users to decide on the trade-off between computational time, computational resources required to solve higher-dimensional nonlinear function abstractions and the tightness of the resulting abstraction.}

\begin{table}[t]
	\caption{Effects of Sample Size on Performance}\vspace{-0.15cm}
	\label{tab:Comp1}
	\begin{center}
		\begin{tabular}{|c|cc|c|c|}
			\hline
			\multicolumn{1}{c}{} & &    \\[\dimexpr-\normalbaselineskip-\arrayrulewidth]
			Performance  & \multicolumn{2}{c|}{Incremental}  & 1-Step & 1-Step \\
			 Parameter & \multicolumn{2}{c|}{Abstraction}  & Abstraction \cite{singh2018mesh}  & Abstraction \cite{Alimguzhin2017}\\
			\hline
			$\overline{s}$ 		& \multicolumn{1}{c|}{50} 	& \multicolumn{1}{c|}{500} 	& All Points 	& All Points\\
			\hline
			Time Taken (sec) 	& \multicolumn{1}{c|}{15} 	& \multicolumn{1}{c|}{6.21} 	& 0.334 & 0.348		\\
			\hline
			max($\theta$) 		& \multicolumn{1}{c|}{300.4}	& \multicolumn{1}{c|}{112.4}	& 80.23 & 84.19		\\
			\hline
		\end{tabular}
	\end{center} \vspace{-0.175cm}
\end{table}

\paragraph{\yongn{Effects of sample size on abstraction error performance with limited memory}} Next, we consider the limited memory case by 
setting the maximum available system RAM to 500MB. 
Here, in each dimension, 5 grid points are chosen, so, depending on the dimension $d$ of the domain, the total number of points will be $5^d$. For incremental abstraction, the maximum number of grid points to take in each increment is set to be $\bar{s}=10^5$ points. Under these resources limitation, the comparison between incremental abstraction and the 1-step abstraction in \cite{singh2018mesh} is summarized in Table \ref{tab:Comp2}. \yongn{We observed that with incremental abstraction, abstractions of higher dimensional nonlinear systems using only limited resources can be achieved with more time (which, as above-mentioned, is less of a concern because the linear programs are solved offline), whereas the 1-step abstraction methods in \cite{singh2018mesh,Alimguzhin2017} return an error and cannot compute any abstraction for $d\ge 8$. Further, the results suggest that given more time, even higher dimensional abstraction problems than are depicted in Table \ref{tab:Comp2} can be solved by computers with limited memory.} 


\begin{table}[t]
	\caption{Performance Under Limited Resources}
	\label{tab:Comp2}
	\begin{center} \vspace{-0.15cm}
		\begin{tabular}{|c|cc|cc|}
			\hline
			\multirow{2}{*}{Dimension} & \multicolumn{2}{c|}{Time Taken (sec.)} & \multicolumn{2}{c|}{Separation} \\\cline{2-5}
			
			 & \multicolumn{1}{c|}{Incremental} & 1-Step & \multicolumn{1}{c|}{Incremental} & 1-Step \\
			 \hline
			 1 &	\multicolumn{1}{c|}{2.091} &	2.179 &	\multicolumn{1}{c|}{55.8} &	55.8 \\
			 \hline
			 3 &	\multicolumn{1}{c|}{2.216} &	2.145 &	\multicolumn{1}{c|}{167.5} &	167.5 \\
			 \hline
			 5 &	\multicolumn{1}{c|}{2.26} &	2.189 &	\multicolumn{1}{c|}{279.2} &	279.2 \\
			 \hline
			 7 &	\multicolumn{1}{c|}{4.927} &	4.393 &	\multicolumn{1}{c|}{390.9} &	390.9 \\
			 \hline
			 9 &	\multicolumn{1}{c|}{69.286}	 & N/A &	\multicolumn{1}{c|}{867.1} &	N/A \\
			 \hline
			 11 & \multicolumn{1}{c|}{2329.178} &	N/A &	\multicolumn{1}{c|}{1659.2} &	N/A \\
			 \hline
			 12 &	\multicolumn{1}{c|}{10095.77} &	N/A &	\multicolumn{1}{c|}{1637.8} &	N/A \\
			 \hline
		\end{tabular}
	\end{center}\vspace{-0.15cm}
\end{table}


\paragraph{\yongn{Effects of heuristics  on abstraction error performance with limited memory}}
\yongn{Additionally, we observed that heuristics can improve the performance of our incremental abstraction in terms of decreased abstraction error.} To better visualize the effects of the heuristics, 
we consider the example with \eqref{eq:rastrigin} in 1D. The example has $s_{max} = 250$  grid points and the maximum number of points $\overline{s}$ is set to 40. 


From our analysis, two major reasons are associated with increased suboptimality of the incremental procedure: (i) conservative approximations due to constraints in \eqref{eq:extension_constr} for guaranteeing future abstractions, and (ii) when using expanding operating region, we will 
start from a closely located cluster of samples, the abstraction of which, for very small $\overline{s}$, may have higher slope than the Lipschitz constant of the system in \eqref{eq:nonl_sys}. 
\yongn{Thus, we conjecture that} one of the ways to tackle the first issue is by choosing the starting region $\mathcal{R}_1$ {smartly}. In Figures \ref{fig:Base_IA}--\ref{fig:Centered_IA}, we show the effects of selecting different starting points on the final abstraction for \eqref{eq:rastrigin} in 1D. By choosing the starting region at the center of the domain $\mathcal{X}$, the overall abstraction is less conservative than the one obtained when the starting region is on one end of the domain as in Figure \ref{fig:Base_IA}. 
Further, \yongn{we conjecture that the second issue} can be resolved by picking sample points that are more spread-out in the domain as a warm-start for the incremental abstraction. This will prevent the closely clustered region to be formed in $\mathcal{R}_1$. In Figure \ref{fig:WS_IA}, providing a random grid point at $x=0.5$ as a warm-start also results in better abstraction than the one obtained without any warm-starts. Instead of 
random samples, certain properties of the nonlinear function $f(x,u)$ 
also can be used for warm-starting, 
e.g., 
global minima or global maxima of $f(x,u)$. 

\yongn{\vspace{-0.05cm}
\subsection{Rendezvous of a Robot Swarm}\vspace{-0.05cm}

			

			

\begin{table}[t]
	\caption{Performance of Abstraction of Swarm Dynamics}
	\label{tab:Comp3}
	\begin{center} \vspace{-0.15cm}
		\begin{tabular}{|c|c|cc|cc|}
			\hline
			\multirow{2}{*}{Agents} & \multirow{2}{*}{State} & \multicolumn{2}{c|}{Time Taken (sec.)} & \multicolumn{2}{c|}{Separation} \\\cline{3-6}
			
			& & \multicolumn{1}{c|}{Incremental} & 1-Step & \multicolumn{1}{c|}{Incremental} & 1-Step \\
			\hline
			\multirow{3}{*}{3} & $f^{\mathbf{x}}_i(x)$ &	\multicolumn{1}{c|}{5.05} &	5.5 &	\multicolumn{1}{c|}{0.1118} &	0.1118 \\
			\cline{2-6} & $f^{\mathbf{y}}_i(x)$ & \multicolumn{1}{c|}{4.73} & \multicolumn{1}{c|}{4.66} & \multicolumn{1}{c|}{0.8798} & \multicolumn{1}{c|}{0.8798}\\
			\cline{2-6} & $f^{\theta}_i(x)$ & \multicolumn{1}{c|}{6.81} & \multicolumn{1}{c|}{5.24} & \multicolumn{1}{c|}{2.9157} & \multicolumn{1}{c|}{2.9157}\\ \hline
			\multirow{3}{*}{5} & $f^{\mathbf{x}}_i(x)$ &	\multicolumn{1}{c|}{1909.85} &	N/A &	\multicolumn{1}{c|}{0.1397} &	N/A \\
			\cline{2-6} & $f^{\mathbf{y}}_i(x)$ & \multicolumn{1}{c|}{1780.72} & \multicolumn{1}{c|}{N/A} & \multicolumn{1}{c|}{1.2437} & \multicolumn{1}{c|}{N/A}\\
			\cline{2-6} & $f^{\theta}_i(x)$ & \multicolumn{1}{c|}{2004.61} & \multicolumn{1}{c|}{N/A} & \multicolumn{1}{c|}{25.5508} & \multicolumn{1}{c|}{N/A}\\
			\hline
		\end{tabular}
	\end{center}\vspace{-0.15cm}
\end{table}

We consider the dynamics of a swarm of robots described in \cite{nagavalli2017automated}, in the form of \eqref{eq:nonl_sys}, with the following parameters: $n=3N$, where $N$ is the number of agents/robots, $m=0$ and $x=\begin{bmatrix} x^\top_1 & \dots & x^\top_N \end{bmatrix}^\top \in \mathbb{R}^n $, where $x$ is the augmented state of the whole swarm, consisting of $x_i$'s, which is the state vector of the agent/robot $i$. Moreover, $x_i =\begin{bmatrix} \mathbf{x}_i & \mathbf{y}_i & \mathbf{\theta}_i \end{bmatrix}^\top \in \mathbb{R}^3$, where $\mathbf{x}_i$, $\mathbf{y}_i$ and $\mathbf{\theta}_i$ are the robot $i$'s $x$-coordinate, $y$-coordinate and heading angle, respectively. Similarly,  $f=\begin{bmatrix} f^\top_1 & \dots & f^\top_N \end{bmatrix}^\top$, where $\forall i \in \{1\dots N\}$, $f_i(.)$ describes the dynamics of robot $i$ as follows: $f_i(.)=\begin{bmatrix} f^{\mathbf{x}}_i(.) & f^{\mathbf{y}}_i(.) & f^{\mathbf{\theta}}_i(.) \end{bmatrix}^\top : \mathbb{R}^n \to \mathbb{R}^3$, with 
$f^{\mathbf{x}}_i(x)=\mathbf{u}^i_v \cos(\mathbf{\theta}_i)$, 
$f^{\mathbf{y}}_i(x)=\mathbf{u}^i_v \sin(\mathbf{\theta}_i)$, 
$f^{\mathbf{\theta}}_i(x)=\mathbf{u}^i_w$, 
where $\mathbf{u}^i_v=b^{i\top}\dot{p}^i$ and $\mathbf{u}^i_w =\phi(b^i,\dot{p}^i)$ \shen{are control inputs forcing each robot to move towards each other}, 
$b^i=\begin{bmatrix} \cos \mathbf{\theta}_i & \sin \mathbf{\theta}_i \end{bmatrix}^\top$ is the ``bearing" vector for the robot $i$, $\dot{p}^i=\frac{1}{\mathcal{N}(i)} \sum_{j \in \mathcal{N}(i)} (p^j-p^i)$, $p^i=\begin{bmatrix} \mathbf{x}_i &\mathbf{y}_i  \end{bmatrix}^\top$ is the ``position" vector of the robot $i$, the function $\phi(v_1,v_2)=\text{sgn}((v_1 \times v_2)^\top \hat{e}_z) \cos ^{-1} (\frac{v^\top_1v_2}{\|v_1\|_2\|v_2\|_2})$ finds the smallest angle required to rotate from vector $v_1$ to vector $v_2$, $\forall v_1,v_2 \in \mathbb{R}^2$ and $\mathcal{N}_i$ is the set of agents in the neighborhood of agent $i$. 
}

\shen{
In this simulation, we consider swarms with $N=3$ and $N=5$ robots\footnote[1]{The states are bounded as $x_1\in [-5,5]$, $x_2 \in [-5,5]$, $x_3 \in [-7,7]$, $x_4 \in [-7,7]$, $x_5 \in [-7,7]$, $y_1\in [0,0.4]$, $y_2 \in [0.5,0.9]$, $y_3 \in [1,5]$, $y_4 \in [0,0.876]$, $y_5 \in [0,1.67]$ and $\theta_i \in [-0.02, 0.02], \forall i \in \{1, \ldots, 5\}$.}, which correspond to \hasn{7-dimensional and 12-dimensional nonlinear systems, respectively.} 
The maximum available system RAM is set to be 500MB.
As shown in Table \ref{tab:Comp3}, both the proposed incremental abstraction and the the 1-step abstraction in \cite{singh2018mesh} can obtain comparable results in terms of computational time and abstraction error for the swarm with 3 robots. However, for more complex swarm with 5 robots, the 1-step abstraction \cite{singh2018mesh} is not able to generate an affine abstraction due to the limited memories, while the proposed incremental approach can still compute it. 
}

\vspace{-0.1cm}
\section{Conclusions} \vspace{-0.05cm}
In this paper, an incremental affine abstraction approach is proposed to simplify a class of nonlinear systems as affine systems, in the sense that two affine hyperplanes are updated dynamically to envelop the nonlinear systems with expanding operating regions. Initially, we consider a small operating region and solve a linear programming to obtain two affine hyperplanes that locally over-approximate the nonlinear system. Then, expanding the operating region with new grid points incrementally, we can find the corresponding affine hyperplanes for a larger domain until the entire domain is covered. The proposed incremental abstraction approach has the capability of reducing the computational space complexity, especially when the nonlinear system has high dimensions. Simulation results are provided to demonstrate the effectiveness of the proposed abstraction method. \shen{Future work will include  
the comparison of the proposed incremental abstraction approach with symbolic approaches in the context of reachability analysis and control synthesis.} 

\vspace{-0.cm}
\bibliographystyle{IEEEtran}
\bibliography{biblio}

\begin{thebibliography}{10}
\providecommand{\url}[1]{#1}
\csname url@rmstyle\endcsname
\providecommand{\newblock}{\relax}
\providecommand{\bibinfo}[2]{#2}
\providecommand\BIBentrySTDinterwordspacing{\spaceskip=0pt\relax}
\providecommand\BIBentryALTinterwordstretchfactor{4}
\providecommand\BIBentryALTinterwordspacing{\spaceskip=\fontdimen2\font plus
\BIBentryALTinterwordstretchfactor\fontdimen3\font minus
  \fontdimen4\font\relax}
\providecommand\BIBforeignlanguage[2]{{%
\expandafter\ifx\csname l@#1\endcsname\relax
\typeout{** WARNING: IEEEtran.bst: No hyphenation pattern has been}%
\typeout{** loaded for the language `#1'. Using the pattern for}%
\typeout{** the default language instead.}%
\else
\language=\csname l@#1\endcsname
\fi
#2}}

\bibitem{pappas2002}
G.~J. Pappas and S.~Simic, ``Consistent abstractions of affine control
  systems,'' \emph{IEEE Transactions on Automatic Control}, vol.~47, no.~5, pp.
  745--756, 2002.

\bibitem{Tabuada2009}
P.~Tabuada, \emph{Verification and control of hybrid systems: a symbolic
  approach}.\hskip 1em plus 0.5em minus 0.4em\relax Springer, 2009.

\bibitem{asarin2007hybridization}
E.~Asarin, T.~Dang, and A.~Girard, ``Hybridization methods for the analysis of
  nonlinear systems,'' \emph{Acta Informatica}, vol.~43, no.~7, pp. 451--476,
  2007.

\bibitem{Asarin2003}
------, ``Reachability analysis of nonlinear systems using conservative
  approximation,'' in \emph{Int. Workshop on Hybrid Systems: Computation and
  Control}.\hskip 1em plus 0.5em minus 0.4em\relax Springer, 2003, pp. 20--35.

\bibitem{Girard2012}
A.~Girard and S.~Martin, ``Synthesis for constrained nonlinear systems using
  hybridization and robust controller on symplices,'' \emph{IEEE Trans. on
  Automatic Control}, vol.~57, no.~4, pp. 1046--1051, 2012.

\bibitem{singh2018mesh}
K.~R. Singh, Q.~Shen, and S.~Z. Yong, ``Mesh-based affine abstraction of
  nonlinear systems with tighter bounds,'' in \emph{2018 IEEE Conference on
  Decision and Control (CDC)}.\hskip 1em plus 0.5em minus 0.4em\relax IEEE,
  2018, pp. 3056--3061.

\bibitem{Singh2018}
K.~R. Singh, Y.~Ding, N.~Ozay, and S.~Z. Yong, ``Input design for nonlinear
  model discrimination via affine abstraction,'' \emph{IFAC-PapersOnLine},
  vol.~51, no.~16, pp. 175--180, 2018.

\bibitem{Alimguzhin2017}
V.~Alimguzhin, F.~Mari, I.~Melatti, I.~Salvo, and E.~Tronci, ``Linearizing
  discrete-time hybrid systems,'' \emph{IEEE Transactions on Automatic
  Control}, vol.~62, no.~10, pp. 5357--5364, 2017.

\bibitem{shen2019}
Q.~Shen and S.~Z. Yong, ``Robust optimization-based affine abstractions for
  uncertain affine dynamics,'' in \emph{2019 American Control Conference
  (ACC)}.\hskip 1em plus 0.5em minus 0.4em\relax IEEE, 2019, pp. 2452--2457.

\bibitem{Jin2019CDC}
Z.~Jin, Q.~Shen, and S.~Z. Yong, ``Optimization-based approaches for affine
  abstraction and model discrimination of uncertain nonlinear systems,'' in
  \emph{IEEE Conference on Decision and Control}, 2019.

\bibitem{coogan2015efficient}
S.~Coogan and M.~Arcak, ``Efficient finite abstraction of mixed monotone
  systems,'' in \emph{Proceedings of the 18th International Conference on
  Hybrid Systems: Computation and Control}, 2015, pp. 58--67.

\bibitem{pola2008approximately}
G.~Pola, A.~Girard, and P.~Tabuada, ``Approximately bisimilar symbolic models
  for nonlinear control systems,'' \emph{Automatica}, vol.~44, no.~10, pp.
  2508--2516, 2008.

\bibitem{reissig2016feedback}
G.~Reissig, A.~Weber, and M.~Rungger, ``Feedback refinement relations for the
  synthesis of symbolic controllers,'' \emph{IEEE Transactions on Automatic
  Control}, vol.~62, no.~4, pp. 1781--1796, 2016.

\bibitem{zamani2014symbolic}
M.~Zamani, P.~Esfahani, R.~Majumdar, A.~Abate, and J.~Lygeros, ``Symbolic
  control of stochastic systems via approximately bisimilar finite
  abstractions,'' \emph{IEEE Transactions on Automatic Control}, vol.~59,
  no.~12, pp. 3135--3150, 2014.

\bibitem{Azuma2010}
S.-I. Azuma, J.-I. Imura, and T.~Sugie, ``Lebesgue piecewise affine
  approximation of nonlinear systems,'' \emph{Nonlinear Analysis: Hybrid
  Systems}, vol.~4, no.~1, pp. 92--102, 2010.

\bibitem{Stampfle2000}
M.~St\"{a}mpfle, ``Optimal estimates for the linear interpolation error for
  simplices,'' \emph{Jour. of Approximation Theory}, vol. 103, pp. 78--90,
  2000.

\bibitem{pohlheim2005geatbx}
H.~Pohlheim, ``Geatbx examples examples of objective functions,'' 2005,
  available: http://www.geatbx.com/.

\bibitem{gurobi}
\BIBentryALTinterwordspacing
{Gurobi Optimization, Inc.}, ``Gurobi optimizer reference manual,'' 2015.
  [Online]. Available: \url{http://www.gurobi.com}
\BIBentrySTDinterwordspacing

\bibitem{nagavalli2017automated}
S.~Nagavalli, N.~Chakraborty, and K.~Sycara, ``Automated sequencing of swarm
  behaviors for supervisory control of robotic swarms,'' in \emph{IEEE Int.
  Conf. on Robotics and Automation}, 2017, pp. 2674--2681.

\end{thebibliography}

\end{document}